\documentclass[10pt]{bmc_article}

\usepackage{cite} % Make references as [1-4], not [1,2,3,4]
\usepackage{url}  % Formatting web addresses
\usepackage{ifthen}  % Conditional
\usepackage{multicol}   %Columns
\usepackage[utf8]{inputenc} %unicode support
\usepackage{amsmath}
\usepackage{amssymb}
\usepackage{amsmath,amssymb,amsfonts,amsthm,graphics,
latexsym, amscd, amsfonts, epsfig, eepic,epic,psfrag,setspace,eufrak}
\usepackage{mathrsfs}
\usepackage{color}
\usepackage{eucal}%    caligraphic-euler fonts: \mathcal{ }
\usepackage{eufrak}%   frak-euler        fonts: \mathfrak{ }
\usepackage[all]{xypic}
\usepackage{xspace}
\usepackage{algorithm}
\usepackage{algorithmic}

\theoremstyle{plain}
\newtheorem{theorem}{Theorem}

\newtheorem{lemma}{Lemma}

\theoremstyle{definition}
\newtheorem{definition}{Definition}

\theoremstyle{remark}

\newboolean{publ}

\newenvironment{bmcformat}{\baselineskip20pt\sloppy\setboolean{publ}{false}}{\baselineskip20pt\sloppy}

\begin{document}
\begin{bmcformat}


\title{On the combinatorics of sparsification}

%%%%%%%%%%%%%%%%%%%%%%%%%%%%%%%%%%%%%%%%%%%%%%
%%                                          %%
%% Enter the authors here                   %%
%%                                          %%
%% Ensure \and is entered between all but   %%
%% the last two authors. This will be       %%
%% replaced by a comma in the final article %%
%%                                          %%
%% Ensure there are no trailing spaces at   %%
%% the ends of the lines                    %%     	
%%                                          %%
%%%%%%%%%%%%%%%%%%%%%%%%%%%%%%%%%%%%%%%%%%%%%%

\author{Fenix W.D. Huang$^1$%
         \email{Fenix W.D. Huang - fenixprotoss@gmail.com}
       and
         Christian M. Reidys\correspondingauthor$^1$%
         \email{Christian M. Reidys\correspondingauthor - duck@santafe.edu}%
      }

%%%%%%%%%%%%%%%%%%%%%%%%%%%%%%%%%%%%%%%%%%%%%%
%%                                          %%
%% Enter the authors' addresses here        %%
%%                                          %%
%%%%%%%%%%%%%%%%%%%%%%%%%%%%%%%%%%%%%%%%%%%%%%

\address{%
    \iid(1)Department of Mathematic and Computer science, University of
    Southern Denmark, Campusvej 55, DK-5230 Odense M, Denmark
}%

\maketitle

%%%%%%%%%%%%%%%%%%%%%%%%%%%%%%%%%%%%%%%%%%%%%%
%%                                          %%
%% The Abstract begins here                 %%
%%                                          %%
%% Please refer to the Instructions for     %%
%% authors on http://www.biomedcentral.com  %%
%% and include the section headings         %%
%% accordingly for your article type.       %%
%%                                          %%
%%%%%%%%%%%%%%%%%%%%%%%%%%%%%%%%%%%%%%%%%%%%%%

\begin{abstract}
{\bf Background:} We study the sparsification of dynamic programming folding
algorithms of RNA structures. Sparsification applies to the mfe-folding of
RNA structures and can lead to a significant reduction of time complexity.

{\bf Results:}
We analyze the sparsification of a particular decomposition rule,
$\Lambda^*$, that splits an interval for RNA secondary and pseudoknot
structures of fixed topological genus.
Essential for quantifying the sparsification is the size of its so
called candidate set. We present a combinatorial framework which
allows by means of probabilities of irreducible substructures to
obtain the expected size of the set of $\Lambda^*$-candidates.
We compute these expectations for arc-based energy models via
energy-filtered generating functions (GF) for RNA secondary structures
as well as RNA pseudoknot structures.
For RNA secondary structures we also consider a simplified loop-energy
model.
This combinatorial analysis is then compared to the expected number
of $\Lambda^*$-candidates obtained from folding mfe-structures.
In case of the mfe-folding of RNA secondary structures with a simplified
loop energy model our results imply that sparsification provides a
reduction of time complexity by a constant factor of $91\%$ (theory)
versus a $96\%$ reduction (experiment).
For the ``full'' loop-energy model there is a reduction of
$98\%$ (experiment).

{\bf Conclusions:}
Our result show that the polymer-zeta property, describing the probability
of an irreducible structure over an interval of length $m$ does not hold
for RNA structures. As a result sparsification of the
$\Lambda^*$-decomposition rule does not lead to a linear reduction of the
set of candidates. We show that under general assumptions the expected number
of $\Lambda^*$-candidates is $\Theta(n^2)$, the constant reduction being in
the range of $95\%$.
The sparsification of the $\Lambda^*$-decomposition rule for RNA
pseudoknotted structures of genus $1$ leads to an expected number
of candidates of $\Theta(n^2)$.
The effect of sparsification is sensitive to the employed energy model.
\end{abstract}

\ifthenelse{\boolean{publ}}{\begin{multicols}{2}}{}

%%%%%%%%%%%%%%%%%%%%%%%%%%%%%%%%%%%%%%%%%%%%%%
%%                                          %%
%% The Main Body begins here                %%
%%                                          %%
%% Please refer to the instructions for     %%
%% authors on:                              %%
%% http://www.biomedcentral.com/info/authors%%
%% and include the section headings         %%
%% accordingly for your article type.       %%
%%                                          %%
%% See the Results and Discussion section   %%
%% for details on how to create sub-sections%%
%%                                          %%
%% use \cite{...} to cite references        %%
%%  \cite{koon} and                         %%
%%  \cite{oreg,khar,zvai,xjon,schn,pond}    %%
%%  \nocite{smith,marg,hunn,advi,koha,mouse}%%
%%                                          %%
%%%%%%%%%%%%%%%%%%%%%%%%%%%%%%%%%%%%%%%%%%%%%%

%%%%%%%%%%%%%%%%
%% Background %%
%%

\section*{Background}

An RNA sequence is a linear, oriented sequence of the nucleotides (bases)
{\bf A,U,G,C}. These sequences ``fold'' by establishing bonds between
pairs of nucleotides. Bonds cannot form arbitrarily: a nucleotide can
at most establish one Watson-Crick base pair {\bf A-U} or {\bf G-C} or a
wobble base pair {\bf U-G}, and the global conformation of an RNA molecule
is determined by topological constraints encoded at the level of secondary
structure, i.e., by the mutual arrangements of the base pairs \cite{Bailor:10}.

Secondary structures can be interpreted as (partial) matchings in a graph of
permissible base pairs \cite{Tabaska:98}. They can be represented as diagrams,
i.e.~graphs over the vertices $1,\dots,n$, drawn on a horizontal line with
bonds (arcs) in the upper halfplane. In this representation one refers to a
secondary structure without crossing arcs as a {\it simple} secondary
structure and pseudoknot structure, otherwise, see Figure~\ref{F:RNAp}.

\begin{figure}[ht]
\begin{center}
\includegraphics[width=0.9\columnwidth]{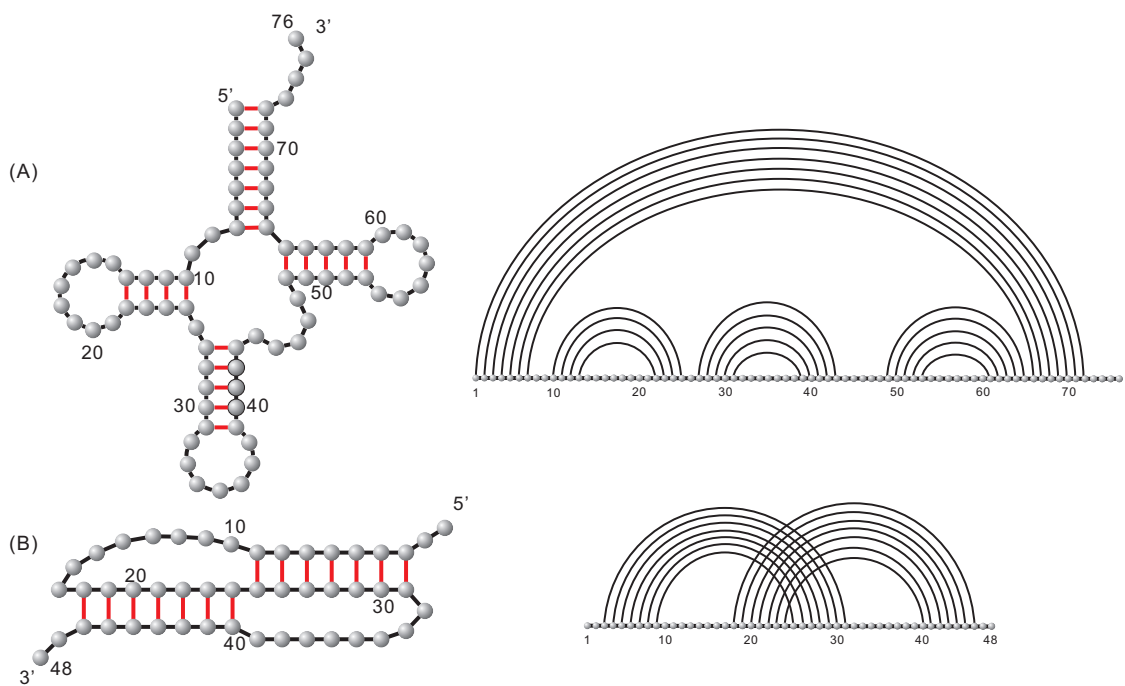}
\end{center}
\caption{\small RNA structures as planar graphs and diagrams.
(A) an RNA secondary structure and (B) an RNA pseudoknot structure. }
\label{F:RNAp}
\end{figure}

Folded configurations are energetically somewhat optimal. Here energy means
free energy, which is dominated by the loops forming between adjacent base
pairs and not by the hydrogen bonds of the individual base pairs
\cite{Mathews:99}.
In addition sterical constraints imply certain minimum
arc-length conditions for minimum free energy configurations
\cite{Waterman:78aa}.
In particular, only configurations without isolated bonds and without bonds
of length one (formed by immediately subsequent nucleotides) are observed in
RNA structures.
In this paper, optimize a problem we meas maximize
the score but not to minimize the free energy.

For a given RNA sequence polynomial-time dynamic programming (DP) algorithms
can be devised, finding such minimal energy configurations. The most commonly
used tools predicting simple RNA secondary structure \texttt{mfold}
\cite{Zuker:89} and the \texttt{Vienna RNA Package} \cite{Hofacker:94a},
are running at $O(N^2)$ space and $O(N^3)$ time solution.
In the following we omit ``simple'' and refer to secondary structures
containing crossing arcs as pseudoknot structures.

Generalizing the matrices of the DP-routines of secondary structure
folding \cite{Zuker:89,Hofacker:94a} to gap-matrices \cite{Rivas:99},
leads to a DP-folding of pseudoknotted structures \cite{Rivas:99}
(\texttt{pknot-R\&E})
with $O(n^4)$ space an $O(n^6)$ time complexity.
The following references provide a certainly incomplete list of
DP-approaches to RNA pseudoknot structure prediction using various
structure classes characterized in terms of recursion equations and/or
stochastic grammars:
\cite{Rivas:99,Uemura:99,Akutsu:00,Lynsoe:00,Cai:03,Dirks:03,Deogun:04,
Reeder:04,Li:05,Matsui:05,Kato:06,Chen:09,Reidys:11a}.
The most efficient algorithm for pseudoknot structures is \cite{Reeder:04}
(\texttt{pknotsRG}) having $O(n^2)$ space and $O(n^4)$ time complexity.
This algorithm however considers only a few types of pseudoknots.

RNA secondary structures are exactly structures of topological genus zero
\cite{Waterman:78a}. The topological classification of RNA structures
\cite{Orland:02,Bon:08,rnag3} has recently been translated into an
efficient DP algorithm \cite{Reidys:11a}. Fixing the topological genus
of RNA structures implies that there are only finitely many types, the
so called irreducible shadows \cite{rnag3}.

Sparsification is a method tailored to speed up DP-algorithms
predicting mfe-secondary structures \cite{spar:07,Backofen:11}.
The idea is to prune certain computation paths encountered in the
DP-recursions, see Figure~\ref{F:spar_sec}.
To make the key point, let us consider the case of RNA secondary structure
folding. Here sparsification reduces the DP-recursion paths to be based on so
called candidates. A candidate is in this case an interval, for which
the optimal solution cannot be written as a sum of optimal solutions
of sub-intervals, see Figure~\ref{F:cand}. Tracing back these candidates gives rise to
``irreducible'' structures and the crucial observation is here that these
irreducibles appear only at a low rate.
This means that there are only relatively few candidates, which in turn
implies a significant reduction in time and space complexity.

Sparsification has been applied in the context of RNA-RNA interaction
structures \cite{Backofen:08} as well as RNA pseudoknot structures
\cite{Backofen:10}. In difference to RNA secondary structures, however,
not every decomposition rule in the DP-folding of RNA pseudoknot
structures is amendable to sparsification.
By construction, sparsification can only be applied for
calculating mfe-energy structures. Since the computation of the partition
function \cite{McCaskill:90,Dirks:03} needs to take into account {\it all}
sub-structures, sparsification does not work.

\begin{figure}[ht]
\begin{center}
\includegraphics[width=0.55\columnwidth]{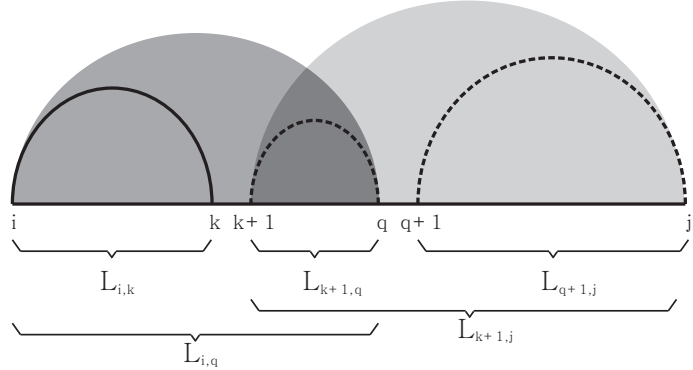}
\end{center}
\caption{\small Sparsification of secondary structure folding.
Suppose the optimal solution $L_{i,j}$ is obtained from the optimal solutions
$L_{i,k}$, $L_{k+1,q}$ and $L_{q+1,j}$. Based on the recursions of the secondary
structures, $L_{i,k}$ and $L_{k+1,q}$ produce an optimal solution of
$L_{i,q}$. Similarly, $L_{k+1,q}$ and $L_{q+1,j}$ produce an optimal solution
of $L_{k+1,j}$.
Now, in order to obtain an optimal solution of $L_{i,j}$ it is sufficient to
consider either the grouping $L_{i,q}$ and $L_{q+1,j}$ or $L_{i,k}$ and
$L_{k+1,j}$.}
\label{F:spar_sec}
\end{figure}

\begin{figure}[ht]
\begin{center}
\includegraphics[width=0.9\columnwidth]{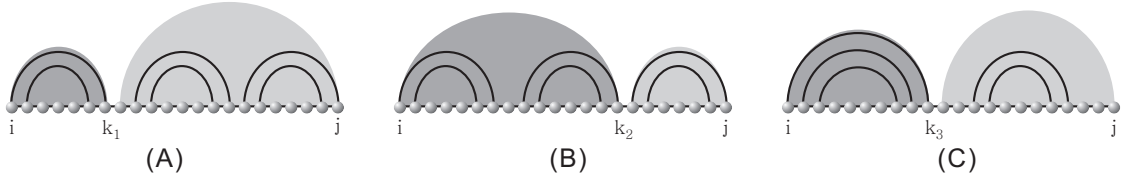}
\end{center}
\caption{\small
What sparsification can and cannot prune: (A) and (B) are two computation
paths yielding the same optimal solution. Sparsification reduces the
computation to path (A) where $S_{i,k_1}$ is irreducible.
(C) is another computation path with distinct leftmost irreducible over
a different interval, hence representing a new candidate that cannot be
reduced to (A) by the sparsification.}
\label{F:cand}
\end{figure}

For the mfe-folding of RNA secondary structures considerable attention has
been paid in order to validate that the set of candidates is small.
The idea here is that irreducibles are contained in short, ``rainbow''-like
arcs. To be precise, the gain is $O(n)$, if secondary structure satisfy
the so called {\it polymer-zeta property} \cite{Kafri:00,Kabakcioglu:05}:
The latter quantifies the probability of an arc of length $m$ to be
$\le b\; m^{-c}$, where $b>0$, $c>1$. Note that these arcs confine in
case of secondary structures irreducible structures, that is arcs and
irreducibility are tightly connected.

In pseudoknotted RNA structures however, we have crossing arcs and the
associated notion of irreducible structures differs significantly from
that of RNA secondary structures.
The polymer-zeta property is theoretically justified by means of modeling the
2D folding of a polymer chain as a self-avoiding walk (SAW) in a 2D lattice
\cite{Venderzande:98}.
More evidence of the polymer-zeta property for RNA secondary structures has
been collected via the NCBI database \cite{NCBI} of mfe-RNA structures.

In this paper we study the sparsification of the decomposition rule $\Lambda^*$
that splices an interval \cite{Backofen:10,Backofen:11} in the context of the
DP-folding of RNA pseudoknot structures of fixed topological genus. Our paper
provides a combinatorial framework to quantify the effects of sparsifying the
$\Lambda^*$-decomposition rule.

We shall prove that the candidate set \cite{spar:07,Backofen:10,Backofen:11}
is indeed small. Our argument is based on assuming a specific distribution of
irreducible structures within mfe-structures. Namely we assume these
irreducibles to appear with probability
${\bf f}^*(n,j)/{\bf f}(n,j)$,
where we assume $e$ to be a fixed parameter and ${\bf F}(z,e)=\sum
{\bf f}_{n,j}z^ne^j$ to be a bivariate (energy-filtered) generating
function whose associated generation function of irreducibles is
${\bf F}^*(z,e)=\sum {\bf f}^*_{n,j}z^ne^j$.

While this energy-filtration seems to be reparameterization of the notion
of ``stickiness'' \cite{Nebel:03}, it is really fundamentally different. This
becomes clear when considering loop-based energies which distinguishes
energy and arcs. Clearly when folding random sequences one weights the
latter around $6/16$, reminiscent of the probability of two given
positions to be compatible. The energy however is fairly independent
as it really depends on the particular loop-type.

We obtain these energy-filtered GFs also for RNA
pseudoknot structures of fixed topological genus. This provides new
insights into the improvements of the sparsification of the
concatenation-rule $\Lambda^*$ in the presence of cross serial
interactions.
Our observations complement the detailed analysis of Backofen
\cite{Backofen:10, Backofen:11}. We show that although for pseudoknot
structures of fixed topological genus \cite{Bon:08,rnag3} the effect of
sparsification on the global time complexity is still unclear, the
decomposition rule that splits an interval can be sped up significantly.

%%%
%%%%%%%%%%%%%%%%%%%%%%%%%%%%%%%%%%%%%%%%%%%%%%%%%%%%%%%%%%%%%%%%%%%%%%%%%%%%
%%%
\subsection*{Sparsification}
%%%
%%%%%%%%%%%%%%%%%%%%%%%%%%%%%%%%%%%%%%%%%%%%%%%%%%%%%%%%%%%%%%%%%%%%%%%%%%%%
%%%

The general idea of sparsification \cite{spar:07,Backofen:10,Backofen:11}
is following:
let ${V}=\{v_1,v_2,\ldots\}$ be a set whose elements $v_i$ are
unions of pairwise disjoint intervals. Let furthermore
$L_v$ denote an optimal solution (a positive number or score) of the DP-routine
over $v$. By assumption $L_v$ is recursively obtained.
Suppose the optimal solution $L_v$ is given by $L_v=L_{v_1}+L_{v_2}+L_{v_3}$,
where $v=v_1\dot\cup v_2\dot\cup v_3$.
Then, under certain circumstances, the DP-routine may interpret
$L_v$ either as $(L_{v_1}+L_{v_2})+L_{v_3}$ or as
$L_{v_1}+(L_{v_2}+L_{v_3})$, see Figure~\ref{F:spar_idea}.
To be precise, this situation is encountered iff
\begin{itemize}
\item there exists an optimal solution $L_{v_1'}$ for a sub-structure
over $v_1'$ where $v_1'=v_1\dot\cup v_2$ via $\Lambda_2$ and $L_v$ is
obtained from $L_{v_1'}$ and $L_{v_3}$ via $\Lambda_1$,
\item there exists an optimal solution $L_{v_2'}$ for a sub-structure over
$v_2'$ where $v_2'=v_2\dot\cup v_3$ via $\Lambda_3$ and $L_v$ is obtained
by $L_{v_1}$ and $L_{v_2'}$ via $\Lambda_1$.
\end{itemize}

Given a decomposition
$$
L_v=\underbrace{\underbrace{L_{v_1}+L_{v_2}}_{\Lambda_2}+L_{v_3}}_{\Lambda_1},
$$
we call $\Lambda_2$ $s$-compatible to $\Lambda_1$ if there exists
a decomposition rule $\Lambda_3$ such that
$$
L_v=\underbrace{L_{v_1}+\underbrace{L_{v_2}+L_{v_3}}_{\Lambda_3}}_{\Lambda_1}.
$$
Note that if $\Lambda_2$ is $s$-compatible to $\Lambda_1$ then
$\Lambda_3$ is $s$-compatible to $\Lambda_1$. To summarize

\begin{definition}{\bf ($s$-compatible)}
Suppose $L_v$ is the optimal solution for $S_v$ over $v$, $L_v=L_{v_1'}+L_{v_3}$
under decomposition rule $\Lambda_1$. $L_{v_1}$ is obtained from two optimal
solutions $L_{v_1}$ and $L_{v_2}$ under rule $\Lambda_2$. Then $\Lambda_2$ is
called {\it $s$-compatible} to $\Lambda_1$ if there exist some rule
$\Lambda_3$ such that $L_{v_2'}=L_{v_2}+L_{v_3}$ and $L_v=L_{v_1}+L_{v_2'}$.
\end{definition}

\begin{figure}[ht]
\begin{center}
\includegraphics[width=0.6\columnwidth]{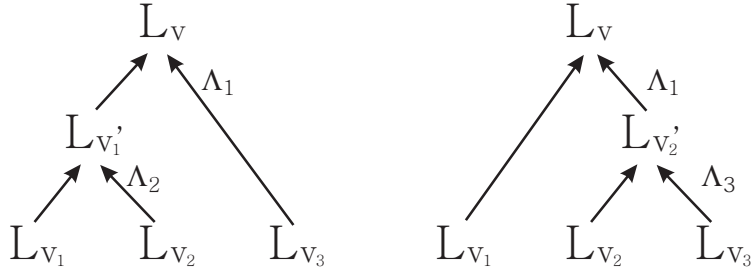}
\end{center}
\caption{\small Sparsification: $L_v$ is alternatively realized via
$L_{v_1}$ and $L_{v_2'}$, or $L_{v_1'}$ and $L_{v_3}$. Thus it is sufficient
to only consider one of the computation paths.}
\label{F:spar_idea}
\end{figure}

Figure~\ref{F:spar_idea} depicts two such ways that realize the same optimal
solution $L_v$. Sparsification prunes any such multiple computations of
the same optimal value.

We next come to the important concept of candidates. The latter mark the
essential computation paths for the DP-routine.

\begin{definition}{\bf (Candidates)}
Suppose $L_v$ is an optimal solution. We call $v$ is a $\Lambda$-{\it candidate}
if for any $v_1 \subsetneq v$ obtained by $\Lambda$ and $v=v_1\dot\cup v_2$,
we have
$$
L_v>L_{v_1}+L_{v_2}
$$
and we shall denote the set of $\Lambda$-candidates set by $Q^{\Lambda}$.
\end{definition}

\begin{lemma}\label{L:11}\cite{spar:07,Backofen:10}
Suppose $\Lambda_2$ is $s$-compatible to $\Lambda_1$ then any optimal solution
$L_v$ can be obtained via $\Lambda_2$-candidates.
\end{lemma}

By construction a $\Lambda_2$-candidate $v$ is a union of disjoint intervals
such that its optimal solution $L_v$ cannot be obtained via a
$\Lambda_2$-splitting. This optimal solution allows to construct a non-unique
arc-configuration (sub-structure) over $v$ \cite{Zuker:89,Hofacker:94a} and
the above $\Lambda_2$-splitting consequently translates into a splitting of
this sub-structure.
This connects the notion of $\Lambda_2$-candidates with that of sub-structures
and shows that a $\Lambda_2$-candidate implies an sub-structure that is
$\Lambda_2$-irreducible.

In the case of sparsification of RNA secondary structures we have one basic
decomposition rule $\Lambda^*$ acting on intervals, namely $\Lambda^*$
splices an interval into two disjoint, subsequent intervals.
The implied notion of a $\Lambda^*$-irreducible sub-structure is that of a
sub-structure nested in an maximal arc, where maximal refers to the
partial order $(i,j)\le (i',j')$ iff $i'\le i\;\wedge j\le j'$.
This observation relates irreducibility to that or arcs and following this
line of thought \cite{spar:07} identifies a specific property of polymer-chains
introduced in \cite{Kafri:00,Kabakcioglu:05} to be of relevance for the size
of candidate sets:

\begin{definition} {\bf (Polymer-zeta property)}\label{D:polymer-zeta}
Let $P(i,j)$ denotes the probability of a structure over an interval
$[i,j]$ under some decomposition rule $\Lambda$. Then we say $\Lambda$
follows the polymer-zeta property if $P(i,j)=bm^{-c}$ for some constant
$b,c>0$.
\end{definition}

This property is theoretically justified by means of modeling the 2D folding
of a polymer chain as a self-avoiding walk (SAW) in a 2D lattice
\cite{Venderzande:98}.

\subsection*{RNA secondary structures}

In this section we recall some results of \cite{spar:07, Backofen:11} on
the sparsification of RNA secondary structures. Secondary structure satisfies
a simple recursion which gives the optimal solution over $[i,j]$ by
$L_{i,j}=\max\{V_{i,j}, W_{i,j}\}$,
where $V_{i,j}$ denotes the optimal solution in which $(i,j)$ is a base pair,
and $W_{i,j}$ denotes the optimal solution obtained by adding the optimal
solutions of two subsequent intervals, respectively. Note that the optimal
solution over a single vertex is denoted by $L_{i,i}$.
We have the recursion equation for $V_{i,j}$ and $W_{i,j}$:
\begin{eqnarray*}
(\Lambda_1)\quad \  V_{i,j}  & = &  L_{i+1,j-1}+ f(i,j), \\
(\Lambda_2) \quad W_{i,j} & = & \max_{i<k <j}\{L_{i,k}+L_{k+1,j}\},
\end{eqnarray*}
where $f(i,j)$ is the score when $(i,j)$ form a base pair, see
Figure.~\ref{F:sec}. In case two positions, $i$,$j$ in the sequence are
incompatible then we have $f(i,j)=-\infty$.

\begin{figure}[ht]
\begin{center}
\includegraphics[width=0.9\columnwidth]{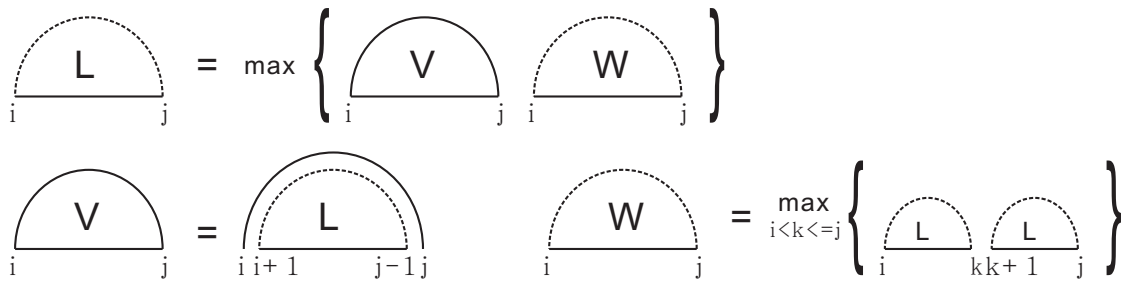}
\end{center}
\caption{\small The recursion solving the optimal solution for
secondary structures.}
\label{F:sec}
\end{figure}
An interval $[i,j]$ is a $\Lambda^*$-candidate if the optimal solution over
$[i,j]$ is given by $L_{i,j}=V_{i,j}>W_{i,j}$. Indeed, $[i,j]$ is a candidate
iff $[i,j]$ is in the candidate set of
$\Lambda^*$, and we denote the set $Q^{\Lambda^*}$ by $Q$.
Suppose the optimal solution $W_{i,j}$ is given by
$W_{i,j}=L_{i,q}+L_{q+1,j}$ and suppose we have $L_{i,q}=L_{i,k}+L_{k+1}$.
Then since $[i,q]$ is not a candidate, Lemma~\ref{L:11} shows that we
can compute $W_{i,j}=L_{i,k}+L_{k+1,j}$, where $[i,k]$ is a candidate.

Accordingly, the recursion for $W_{i,j}$ can be based on candidates,
i.e.~$W_{i,j}=\max_{[i,k]\in Q} \{L_{i,k}+L_{k+1,j}\}$.
Clearly, the bottleneck for computing the recursion is the calculation of
$W_{i,j}$, which requires $O(n^3)$ time. Applying sparsification, this
recursion is based on candidates $[i,k]$. Suppose we have $Z$ such
candidates, then the time complexity reduces to $O(nZ)$, since the
optimal solution is necessarily based on a candidate. Once the latter
is identified the expression $L_{k+1,j}$ requires only $O(n)$ time
complexity. In the worst case, $Q$ contains $O(n^2)$ elements.

The polymer-zeta property however implies that the expectation of $Z$
is given by $\sum_{i\ge 1}^n \sum_{j=i}b(j-i)^{-c}$ where $b$ and $c$ are
constants and $c>1$. We can conclude from the polymer-zeta property
that $Z=O(n)$ and accordingly the runtime reduces to $O(n)
\cdot O(n)=O(n^2)$.

\subsection*{RNA pseudoknot structures}

Sparsification can also be applied to the DP-algorithm folding RNA structures
with pseudoknots \cite{Backofen:10}. In contrast to the decomposition
rule $\Lambda^*$ that spliced an interval into two subsequent intervals,
we encounter in the grammar for pseudoknotted structures additional more
complex decomposition rules \cite{Rivas:99}.
As shown in \cite{Backofen:10} there exist some decomposition rules which
are not $s$-compatible and which can accordingly not be sparsified at all, see Figure~\ref{F:spar_g}.
For instance, given a decomposition rule $\Lambda$ in \texttt{pknot-R\&E}
subsequent decomposition rules which are $s$-compatible to $\Lambda$ are
referred to as split type of $\Lambda$ \cite{Backofen:10}.

In the following we will study RNA pseudoknot structures of fixed topological
genus, see Section~{\bf Diagrams, surfaces and some generating functions} for
details.
An algorithm folding such pseudoknot structures, \texttt{gfold}, has been
presented in \cite{Reidys:11a}. The decomposition rules that appear in
\texttt{gfold} are reminiscent to those of \texttt{pknot-R\&E} but as
they restrict the genus of sub-structures the iteration of gap-matrices
is severely restricted and the effect of sparsification of these
decompositions is significantly smaller.

\begin{figure}[ht]
\begin{center}
\includegraphics[width=0.9\columnwidth]{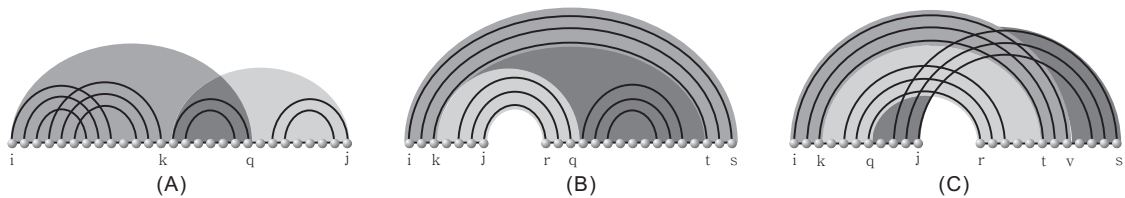}
\end{center}
\caption{\small Decomposition rules for pseudoknot structures of fixed genus.
(A) three decompositions via the rule $\Lambda^*$, which is $s$-compatible to
    itself.
    We show that for $\Lambda^*$ we obtain a linear reduction in time
    complexity.
(B) three decomposition rules $\Lambda_1,\Lambda_2,\Lambda_3$ where
    $\Lambda_2,\Lambda_3$ are $s$-compatible to $\Lambda_1$.
    A quantification of the candidate set is not implied by the
    polymer-zeta property.
(C) three decomposition rules $\Lambda_1,\Lambda_2,\Lambda_3$ where
    $\Lambda_2,\Lambda_3$ are not $s$-compatible to $\Lambda_1$.
}
\label{F:spar_g}
\end{figure}

In the following, we restrict our analysis to the decomposition rule
$\Lambda^*$ which splices an interval into two subsequent intervals.
Expressed in combinatorial language, $\Lambda^*$ cuts the backbone of
an RNA pseudoknot structure of fixed genus $g$ over one interval without
cutting a bond.

%%%
%%%%%%%%%%%%%%%%%%%%%%%%%%%%%%%%%%%%%%%%%%%%%%%%%%%%%%%%%%%%%%%%%%%%%%%%%%%%%%%
%%%

\section*{Methods}

%%%
%%%%%%%%%%%%%%%%%%%%%%%%%%%%%%%%%%%%%%%%%%%%%%%%%%%%%%%%%%%%%%%%%%%%%%%%%%%%%%%
%%%

%%%
%%%%%%%%%%%%%%%%%%%%%%%%%%%%%%%%%%%%%%%%%%%%%%%%%%%%%%%%%%%%%%%%%%%%%%%%%%
%%%
\subsection*{Diagrams and genus filtration}
%%%
%%%%%%%%%%%%%%%%%%%%%%%%%%%%%%%%%%%%%%%%%%%%%%%%%%%%%%%%%%%%%%%%%%%%%%%%%%
%%%

In this section we recall some facts about diagrams and pass from
diagrams to surfaces in order to be able to formulate what we mean
by an RNA pseudoknot structure of fixed genus $g$. Most of this
section is derived from \cite{Zagier:95,rnag3} with the exception of
Lemma~\ref{L:irr_rec} and Theorem~\ref{T:irr}, which are new
and key for the subsequent analysis of $\Lambda^*$-candidates.

A diagram is a labeled graph over the vertex set $[n]=\{1, \dots, n\}$ in
which each vertex has degree $\le 3$, represented by drawing its vertices
in a horizontal line. The backbone of a diagram is the sequence of
consecutive integers $(1,\dots,n)$ together with the edges $\{\{i,i+1\}
\mid 1\le i\le n-1\}$. The arcs of a diagram, $(i,j)$, where $i<j$, are
drawn in the upper half-plane. We shall distinguish the backbone edge
$\{i,i+1\}$ from the arc $(i,i+1)$, which we refer to as a $1$-arc.
A stack of length $\ell$ is a maximal sequence of ``parallel'' arcs,
$((i,j),(i+1,j-1),\dots,(i+(\ell-1),j-(\ell-1)))$
and is also referred to as a $\ell$-stack, see
Figure~\ref{F:diagram}.

\begin{figure}[ht]
\begin{center}
\includegraphics[width=0.6\columnwidth]{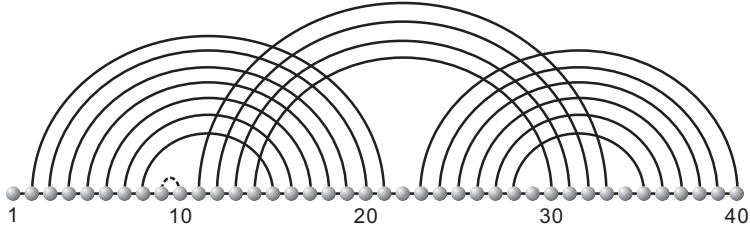}
\end{center}
\caption{\small RNA structures and diagram representation.
A diagram over $\{1,\ldots,40\}$. The arcs $(1,21)$ and
$(11,33)$ are crossing and the dashed arc $(9,10)$ is a $1$-arc
which is not allowed.
This structure contains $3$ stacks with length $7$, $4$ and $6$, from left to
right respectively.
}
\label{F:diagram}
\end{figure}

We shall consider diagrams as fatgraphs, $\mathbb{G}$, that is graphs $G$
together with a collection of cyclic orderings, called fattenings, one such
ordering on the half-edges incident on each vertex.
Each fatgraph $\mathbb{G}$ determines an oriented surface $F(\mathbb{G})$
\cite{Loebl:08,Penner:10} which is connected if $G$ is and has some
associated genus $g(G)\ge 0$ and number $r(G)\ge 1$ of boundary components.
Clearly, $F(\mathbb{G})$ contains $G$ as a deformation retract \cite{Massey:69}.
Fatgraphs were first applied to RNA secondary structures in
\cite{Waterman:93} and \cite{Penner:03}.

A diagram $\mathbb{G}$ hence determines a unique surface $F(\mathbb{G})$
(with boundary). Filling the boundary components with discs we can
pass from $F(\mathbb{G})$ to a surface without boundary. Euler
characteristic, $\chi$, and genus, $g$, of this surface is given by
$\chi =  v - e + r$ and $g  =  1-\frac{1}{2}\chi$, respectively, where
$v,e,r$ is the number of discs, ribbons and boundary components in
$\mathbb{G}$, \cite{Massey:69}.
The genus of a diagram is that of its associated surface without boundary and
a diagram of genus $g$ is referred to as $g$-diagram.

A $g$-diagram without arcs of the form $(i,i+1)$ ($1$-arcs) is called a
$g$-structure. A $g$-diagram that contains only vertices of degree three,
i.e.~does not contain any vertices not incident to arcs in the upper halfplane,
is called a $g$-matching.
A stack of length $\tau$ is a maximal sequence of
``parallel'' arcs,
$$
((i,j),(i+1,j-1),\dots,(i+\tau,j-\tau)).
$$

A diagram is called irreducible, if and only if it cannot be split into two
by cutting the backbone without cutting an arc.

Let ${\bf c}_g(n)$ and ${\bf d}_g(n)$ denote the number of $g$-matchings and
$g$-structures having $n$-arcs and $n$ vertices, respectively, with
GF
$$
{\bf C}_g(z)=\sum_{n=0}^{\infty}{\bf c}_g(n)z^n \quad \quad {\bf D}_g(z)=
\sum_{n=0}^{\infty}{\bf d}_g(n)z^n.
$$
The GF ${\bf C}_g(z)$ has been computed the context of the
virtual Euler characteristic of the moduli-space of curves in \cite{Zagier:95}
and ${\bf D}_g(z)$ can be derived from ${\bf C}_g(z)$ by means of symbolic
enumeration \cite{rnag3}. The GF of genus zero diagrams ${\bf C}_0(z)$ is
wellknown to be the GF of the Catalan numbers, i.e., the numbers of
triangulations of a polygon with $(n+2)$ sides,
$$
{\bf C}_0(z)={{1-\sqrt{1-4z}}\over{2z}}.
$$
As for $g\ge 1$ we have the following situation \cite{rnag3}
%%%
%%%%%%%%%%%%%%%%%%%%%%%%%%%%%%%%%%%%%%%%%%%%%%%%%%%%%%%%%%%%%%%%%%%%%%%%
%%%
\begin{theorem}\label{T:genus}
Suppose $g\geq 1$. Then the following assertions hold\\
{\bf (a)} ${\bf D}_{g}(z)$ is algebraic and
\begin{eqnarray}\label{E:oho}
{\bf D}_{g}(z) & = & \frac{1}{z^2-z+1}\
                            {\bf C}_g\left(\frac{z^2}
                            {\left(z^{2}-z+1\right)^2}\right).
\end{eqnarray}
In particular, we have for some constant $a_g$ depending only on $g$
and $\gamma \approx 2.618$:
\begin{equation}
[z^n]{\bf D}_{g}(z)\sim a_g\,n^{3(g-\frac{1}{2})} \gamma^n.
\end{equation}
{\bf (b)} the bivariate GF of $g$-structures
over $n$ vertices, containing exactly $m$ arcs, ${\bf E}_{g}(z,t)$,
is given by
\begin{equation}\label{E:qwe}
{\bf E}_{g}(z,t) = \frac{1}{t z^2-z+1}{\bf D}_g\left(
\frac{t \; z^2}{(t\; z^2-z+1)^2}\right).
\end{equation}
\end{theorem}
%%%
%%%%%%%%%%%%%%%%%%%%%%%%%%%%%%%%%%%%%%%%%%%%%%%%%%%%%%%%%%%%%%%%%%%%%%%%
%%%

%%%
%%%%%%%%%%%%%%%%%%%%%%%%%%%%%%%%%%%%%%%%%%%%%%%%%%%%%%%%%%%%%%%%%%%%%%%%
%%%
\subsection*{Irreducible $g$-structures}
%%%
%%%%%%%%%%%%%%%%%%%%%%%%%%%%%%%%%%%%%%%%%%%%%%%%%%%%%%%%%%%%%%%%%%%%%%%%
%%%
In the context of $\Lambda^*$-candidates we observed that irreducible
substructures are of key importance. It is accordingly of relevance to
understand the combinatorics of these structures. To this end let
${\bf D}^*_g(z)=\sum_{n=0}^{\infty}{\bf D}^*_g(n)z^n$ denote the GF of
irreducible $g$-structures.

%%%
%%%%%%%%%%%%%%%%%%%%%%%%%%%%%%%%%%%%%%%%%%%%%%%%%%%%%%%%%%%%%%%%%%%%%%%%%%%%%
%%%
\begin{lemma} \label{L:irr_rec}
For $g\ge 0$, the GF ${\bf D}^*_g(z)$ satisfies the
recursion
\begin{eqnarray*}
{\bf D}^*_0(z) & = & 1-\frac{1}{{\bf D}_0(z)} \\
{\bf D}^*_g(z) & = & -\frac{({\bf D}^*_0(z)-1){\bf D}_g(z)+
\sum_{g_1=1}^{g-1}{\bf D}^*_{g_1}(z){\bf D}_{g-g_1}(z)}{{\bf D}_0(z)}.
\end{eqnarray*}
\end{lemma}
%%%
%%%%%%%%%%%%%%%%%%%%%%%%%%%%%%%%%%%%%%%%%%%%%%%%%%%%%%%%%%%%%%%%%%%%%%%%%%%%%
%%%

For a proof of Lemma~\ref{L:irr_rec}, see Section~{\bf Proofs}.

%%%
%%%%%%%%%%%%%%%%%%%%%%%%%%%%%%%%%%%%%%%%%%%%%%%%%%%%%%%%%%%%%%%%%%%%%%%%%%%%%%
%%%
\begin{theorem} \label{T:irr}
For $g\ge 1$ we have\\
{\bf (a)} the GF of irreducible $g$-structures over $n$
vertices is given by
\begin{equation}
{\bf D}^*_g(z)=(z^2-z+1)\left(\frac{{\bf U}_g(u)}{(1-4u)^{3g-\frac{1}{2}}}+
\frac{{\bf V}_g(u)}{(1-4u)^{3g-1}}\right),
\end{equation}
where $u=\frac{z^2}{(z^2-z+1)^2}$, ${\bf U}_g(z)$ and ${\bf V}_g(z)$ are
both polynomials with lowest degree at least $2g$, and ${\bf U}_g(1/4)$,
${\bf V}_g(1/4)\neq 0$. In particular, for some constant $k_g>0$ and
$\gamma\approx 2.618$:
\begin{equation}
{\bf D}^*_g(n) \sim k_g n^{3(g-\frac{1}{2})} \gamma^n.
\end{equation}
{\bf (b)} the bivariate GF of irreducible $g$-structures
over $n$ vertices, containing exactly $m$ arcs, ${\bf E}^*_{g}(z,t)$, is given by
\begin{equation}\label{E:qw}
{\bf E}^*_{g}(z,t) = (t z^2-z+1)\left(\frac{{\bf U}_g(v)}{(1-4v)^{3g-\frac{1}{2}}}+
\frac{{\bf V}_g(v)}{(1-4v)^{3g-1}}\right),
\end{equation}
where $v=\frac{t z^2}{(t z^2-z+1)^2}$.
\end{theorem}

%%%
%%%%%%%%%%%%%%%%%%%%%%%%%%%%%%%%%%%%%%%%%%%%%%%%%%%%%%%%%%%%%%%%%%%%%%%%%%%%%%
%%%

We shall postpone the proof of Theorem~\ref{T:irr} to
Section~{\bf Proofs}.

%%%
%%%%%%%%%%%%%%%%%%%%%%%%%%%%%%%%%%%%%%%%%%%%%%%%%%%%%%%%%%%%%%%%%%%%%%%%%%%%%%%
%%%

\subsection*{The main result}
%%%
%%%%%%%%%%%%%%%%%%%%%%%%%%%%%%%%%%%%%%%%%%%%%%%%%%%%%%%%%%%%%%%%%%%%%%%%%%%%%%%
%%%

In Section~{\bf Sparsification} we observed that sparsification applies to
the decomposition rule $\Lambda^*$, which effectively splices off an
irreducible sub-structure (diagram). This notion of $\Lambda^*$-irreducibility
is indeed compatible by the notion of combinatorial irreducibility introduced
in Section~{\bf Diagrams, surfaces and some generating functions}, see
Figure~\ref{F:diagram10}.
An optimal solution for the original structure is obtained from an optimal
solution of the spliced, $\Lambda^*$-{irreducible}, sub-structure and an
optimal solution for the remaining sub-structure.

\begin{figure}[ht]
\begin{center}
\includegraphics[width=0.7\columnwidth]{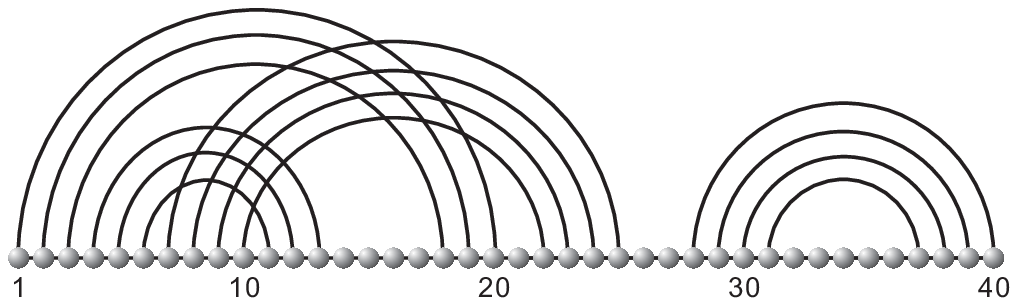}
\end{center}
\caption{\small Irreducibility relative to a decomposition rule:
the rule $\Lambda^*$ splitting $S_{i,j}$ to $S_{i,k}$ and
$S_{k+1,j}$, $S_{1,40}$ is not $\Lambda^*$-irreducible, while $S_{1,25}$ and
$S_{28,40}$ are.
However, for the decomposition rule $\Lambda_2$, which removes the
outmost arc, $S_{28,40}$ is not $\Lambda_2$-irreducible while $S_{1,25}$ is. }
\label{F:diagram10}
\end{figure}

Folded configurations are energetically optimal and dominated by the
stacking of adjacent base pairs \cite{Mathews:99}, as well as minimum
arc-length conditions \cite{Waterman:78aa} discussed before.

In the following we mimic some form of minimum free energy $g$-structures:
inspired by the Nussinov energy model \cite{Nussinov:78} we consider the
weight of a $g$-structure over $n$ vertices to be given by $\eta^\ell$, where
$\ell$ is the number of arcs for some $\eta\ge 1$ \cite{Nebel:03}.
Note that the case $\eta=1$
corresponds to the uniform distribution, i.e.~all $g$-structure have identical
weight.

This approach requires to keep track of the number of arcs, i.e.~we need
to employ bivariate GF. In Theorem~\ref{T:genus} {\bf (b)} we computed
this bivariate GF and in Theorem~\ref{T:irr} {\bf (b)} we derived from
this bivariate GF ${\bf E}^*_g(z,t)$, the GF of irreducible $g$-structures
over $n$ vertices containing $\ell$ arcs.

The idea now is to substitute for the second indeterminant, $t$,
some fixed $\eta\in \mathbb{R}$. This substitution induces the
formal power series
$$
{\bf D}_{g,\eta}(z)={\bf E}_{g}(z,\eta),
$$
which we regard as being parameterized by $\eta$.
Obviously, setting $\eta=1$ we recover ${\bf D}_g(z)$, i.e.~we have
${\bf D}_{g}(z)={\bf D}_{g,1}(z)={\bf E}_{g}(z,1)$. Note that for
$\eta>1/4$, the polynomial $\eta z^2-z+1$ has no real root.
Thus we have for $\eta>1/4$ the asymptotics
\begin{equation} \label{E:eta_sim}
{\bf d}_{g,\eta}(n)\sim a_{g,\eta} n^{3(g-\frac{1}{2})} \gamma_\eta^n \quad
\text{and}\quad
{\bf d}^*_{g,\eta}(n)\sim k_{g,\eta} n^{3(g-\frac{1}{2})} \gamma_\eta^n,
\end{equation}
with identical exponential growth rates as long as the supercritical
paradigm \cite{Flajolet:07a} applies, i.e.~as long as
$\gamma_\eta$, the real root of minimal modulus of
$$
\left(\frac{\eta \; z^2}{(\eta\; z^2-z+1)^2}\right)=\frac{1}{4},
$$
is smaller than any singularity of $\frac{1}{\eta z^2-z+1}$.
In this situation $\eta$ affects the constant $a_{g,\eta}$ and the
exponential growth rate $\gamma_{\eta}$ but {\it not} the sub-exponential
factor $n^{3(g-\frac{1}{2})}$. The latter stems from the singular expansion of
${\bf C}_g(z)$. Analogously, we derive the $\eta$-parameterized family of
GF ${\bf D}^*_{g,\eta}(z)={\bf E}^*_{g}(z,\eta)$.
Assuming a random sequence has on average a probability at most $6/16$ to
form a base pair we fix in the following $\eta=6e/16\approx 1.0125$,
where $e$ is the Euler number. By abuse of notation we will omit the
subscript $\eta$ assuming $\eta=6e/16$.

The main result of this section is that the set of $\Lambda^*$-candidates
is small. To put this size into context we note that the total number of
entries considered for the $\Lambda^*$-decomposition rule is given by
$$
\Omega(n)=\sum_{m=1}^n(n-m+1).
$$
%%%
%%%%%%%%%%%%%%%%%%%%%%%%%%%%%%%%%%%%%%%%%%%%%%%%%%%%%%%%%%%%%%%%%%%%%%%%%
%%%
\begin{theorem}\label{T:eumel1}
Suppose an mfe $g$-structure over an interval of length $m$ is irreducible
with probability ${\bf d}^*_g(m)/{\bf d}_g(m)$, then the expected number
of candidates of $g$-structures for sequences of lengths $n$ satisfies
$$
\mathbb{E}_g(n)= \Theta(n^2)
$$
and furthermore, setting $\overline{\mathbb{E}}_g(n)=\mathbb{E}_g(n) /
\Omega(n)$ we have
$$
\overline{\mathbb{E}}_g(n) \sim  {\bf d}^*_{g}(n)/{\bf d}_g(n) \sim b_g,
$$
where $b_g>0$ is a constant.
\end{theorem}
We provide an illustration of Theorem~\ref{T:eumel1} in Figure~\ref{F:Ecand}.
%%%
%%%%%%%%%%%%%%%%%%%%%%%%%%%%%%%%%%%%%%%%%%%%%%%%%%%%%%%%%%%%%%%%%%%%%%%%%%%%%
%%%

\begin{figure}[ht]
\begin{center}
\includegraphics[width=0.7\columnwidth]{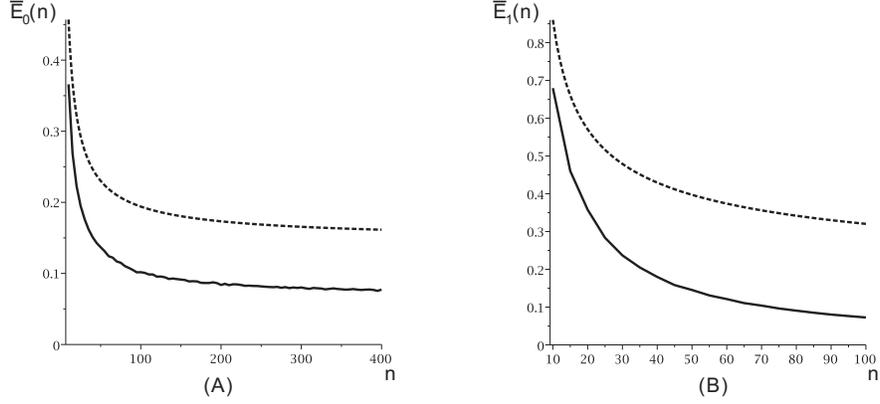}
\end{center}
\caption{\small
The expected number of candidates for secondary and $1$-structures,
$\overline{\mathbb{E}}_{0}(n)$ and $\overline{\mathbb{E}}_{1}(n)$:
we compute the expected number of candidates obtained by folding
$100$ random sequences for secondary structures (A)(solid)
and $1$-structures (B)(solid). We also display the theoretical expectations
implied by Theorem~\ref{T:eumel1} (A)(dashed) and (B)(dashed).
}
\label{F:Ecand}
\end{figure}
%%%
%%%%%%%%%%%%%%%%%%%%%%%%%%%%%%%%%%%%%%%%%%%%%%%%%%%%%%%%%%%%%%%%%%%%%%%%%%%%%%
%%%

\begin{proof}

We proof the theorem by quantifying the probability of $[i,j]$ being a
$\Lambda^*$-candidate. In this case any (not necessarily
unique) sub-structure, realizing the optimal solution $L_{i,j}$, is
$\Lambda^*$-irreducible, and therefore an irreducible structure over $[i,j]$.

Let $m=(j-i+1)$, by assumption, the probability that $[i,j]$ is a
candidate conditional to the existence of a substructure over $[i,j]$
is given by
\begin{equation}\label{E:P1}
\mathbb{P}_{*}([i,j] \mid [i,j] \ \text{\rm is a candidate}\ )=
\frac{{\bf d}^*_{g}(m)}{{\bf d}_{g}(m)},
\end{equation}
Note that $\mathbb{P}_{*}([i,j] \mid [i,j] \ \text{\rm is a candidate}\ )$
does not depend on the relative location of the interval but only on the
interval-length. Let $\mathbb{P}_{g}(m)={\bf d}^*_{g}(m)/{\bf d}_{g}(m)$,
then according to Theorem~\ref{T:genus},
\begin{eqnarray*}
(1-\epsilon) a_{g} m^{3(g-\frac{1}{2})} \gamma^m  & \le &
{\bf d}_{g}(m) \ \le \ (1+\epsilon) a_{g} m^{3(g-\frac{1}{2})} \gamma^m, \\
(1-\epsilon) k_{g} m^{3(g-\frac{1}{2})} \gamma^m  & \le &
{\bf d}^*_{g}(m)\  \le \ (1+\epsilon) k_{g} m^{3(g-\frac{1}{2})} \gamma^m,
\end{eqnarray*}

for $m\ge m_0$ where $m_0>0$ and $0<\epsilon<1$ are constants. On the one
hand
\begin{equation}
\mathbb{P}_{g}(m) = \frac{{\bf d}^*_{g}(m)}{{\bf d}_{g}(m)}
\le \frac{(1+\epsilon)a_{g} m^{3(g-\frac{1}{2})}
\gamma^m}{(1-\epsilon) k_{g} m^{3(g-\frac{1}{2})} \gamma^m}
= (1+\epsilon') \frac{a_{g}}{k_{g}}= (1+\epsilon')b_{g},
\end{equation}
where $b_g=a_g/k_g>0$ is a constant.
On the other hand, we have
\begin{equation}
\mathbb{P}_{g}(m) = \frac{{\bf d}^*_{g}(m)}{{\bf d}_{g}(m)}
\ge \frac{(1-\epsilon)a_{g} m^{3(g-\frac{1}{2})} \gamma^m}{(1+\epsilon)
k_{g} m^{3(g-\frac{1}{2})} \gamma^m}
= (1-\epsilon'') \frac{a_{g}}{k_{g}} = (1-\epsilon'') b_{g}.
\end{equation}
Setting $\epsilon=\max\{\epsilon', \epsilon''\}$, we can conclude that
$\mathbb{P}_{g}(m) \sim {\bf d}^*_g(m)/{\bf d}_g(m)$, see Fig.~\ref{F:prob}.
%%%
%%%%%%%%%%%%%%%%%%%%%%%%%%%%%%%%%%%%%%%%%%%%%%%%%%%%%%%%%%%%%%%%%%%%%%%%%%%%
%%%
\begin{figure}[ht]
\begin{center}
\includegraphics[width=0.7\columnwidth]{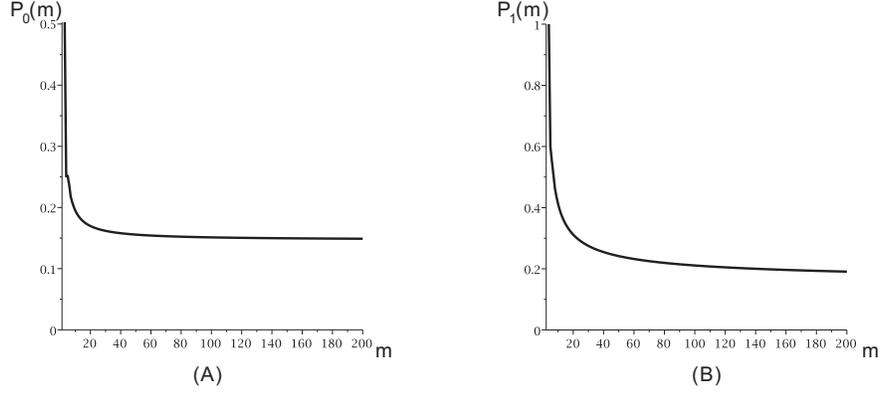}
\end{center}
\caption{\small The probability distribution of $\mathbb{P}_{0}(m)$ (A) and
$\mathbb{P}_{1}(m)$ (B). }
\label{F:prob}
\end{figure}
%%%
%%%%%%%%%%%%%%%%%%%%%%%%%%%%%%%%%%%%%%%%%%%%%%%%%%%%%%%%%%%%%%%%%%%%%%%%%%%%
%%%

We next study the expected number of candidates over an interval of
length $m$. To this end let
$$
X_{m}=
\vert\{[i,j]\mid [i,j] \ \text{\rm is a $\Lambda^*$-candidate of
length $m$} \, \}\vert.
$$
The expected cardinality of the set of $\Lambda^*$-candidates of
length $m=(j-i+1)$ encountered in the DP-algorithm is given by
\begin{eqnarray*}
\mathbb{E}_g(X_{m}) & \le & \, (n-(m-1))\, \mathbb{P}_{g}(m),
\end{eqnarray*}
since there are $n-(m-1)$ starting points for such an interval $[i,j]$.
Therefore, by linearity of expectation, for sufficiently
large $m>m_0$, $\mathbb{P}_{g}(m) \le (1+\epsilon)b_{g}$ with
$\epsilon$ being a small constant. Thus we have
\begin{equation}
 \mathbb{E}_g(n)=\mathbb{E}_g(\sum_m X_{m}) \le  \sum_{m=1}^{m_0} (n-m+1) \mathbb{P}_{g}(m)+
 (1+\epsilon)b_{g} \sum_{m=m_0}^n (n-m+1).
\end{equation}
Consequently, the expected size of the $\Lambda^*$-candidate set is
$\Theta(n^2)$.
We proceed by comparing the expected number of candidates of a sequence
with length $n$ with $\Omega(n)$,
\begin{eqnarray*}
\frac{\mathbb{E}_g(n)}{\Omega(n)} & \le & \frac{\sum_{m=1}^{m_0} (n-m+1) \mathbb{P}_{g}(m)+
 (1+\epsilon)b_{g} \sum_{m=m_0}^n (n-m+1)}{\sum_{m=1}^n(n-m+1)} \\
 & \le & (1+\epsilon)b_{g}+\frac{\sum_{m=1}^{m_0}(\mathbb{P}_{g}(m)-(1+\epsilon)b_{g})(n-m+1)}
 {\sum_{m=1}^n(n-m+1)} \\
 & \le & (1+\epsilon)b_{g}+\frac{k\cdot n}{n^2}.
\end{eqnarray*}
For sufficient large $n\ge n_0$, $\mathbb{E}_g(n)/ \Omega(n)\le
(1+\epsilon')b_{g}$. Furthermore
\begin{eqnarray*}
\frac{\mathbb{E}_g(n)}{\Omega(n)}
\ge  \frac{\sum_{m=1}^{m_0} (n-m+1) \mathbb{P}_{g}(m)+
 (1-\epsilon)b_{g} \sum_{m=m_0}^n (n-m+1)}{\sum_{m=1}^n(n-m+1)}
 \ge  (1-\epsilon)b_{g},
\end{eqnarray*}
from which we can conclude $\mathbb{E}_g(n)/\Omega(n) \sim {\bf d}^*_g(m)/{\bf d}_g(m) \sim b_{g}$
and the theorem is proved.
\end{proof}

%%%
%%%%%%%%%%%%%%%%%%%%%%%%%%%%%%%%%%%%%%%%%%%%%%%%%%%%%%%%%%%%%%%%%%%%%%%%%%%%
%%%
%%%
%%%%%%%%%%%%%%%%%%%%%%%%%%%%%%%%%%%%%%%%%%%%%%%%%%%%%%%%%%%%%%%%%%%%%%%%%%%%
%%%
\subsubsection*{Loop-based energies}
%%%
%%%%%%%%%%%%%%%%%%%%%%%%%%%%%%%%%%%%%%%%%%%%%%%%%%%%%%%%%%%%%%%%%%%%%%%%%%%%
%%%
In this section we discuss the more realistic loop-based energy model of
RNA secondary structure folding. To be precise we evoke here instead of
two trivariate GFs
${\bf F}(z,t,v)$ and ${\bf F}^*(z,t,v)$ counting secondary structures over
$n$ vertices that filter energy and arcs.

This becomes necessary since the loop-based model distinguishes between
arcs and energy. The ``cancelation'' effect or reparameterization of
stickiness \cite{Nebel:03}
to which we referred to before does not appear in this context.
Thus we need both an arc- as well as an energy-filtration.

A further complication emerges. In difference to the GFs ${\bf E}_g(z,t)$ and
${\bf E}^*_g(z,t)$ the new GFs are not simply obtained by formally
substituting $(t z^2/((t z^2-z+1)^2)$ into the power series
${\bf D}_g(z)$ and ${\bf D}^*_g(z)$ as bivariate terms. The more complicated
energy model requires a specific recursion for irreducible secondary
structures.

The energy model used in prediction secondary structure is more
complicated than the simple arc-based energy model. Loops which are
formed by arcs as well as isolated vertexes between the arcs are
considered to give energy contribution.
Loops are categorized as hairpin loops (no nested arcs), interior loops
(including bulge loops and stacks) and multi-loops
(more than two arc nested), see Figure~\ref{F:Loops}.
An arbitrary secondary structure can be uniquely decomposed into a collection
of mutually disjoint loops. A result of the particular energy parameters
\cite{Mathews:99} is that the energy model prefers interior loops, in
particular stacks (no isolated vertex between two parallel arc),
and disfavors multi-loops. Base on this observation, we give a simplified
energy model for a loop $\lambda$ contained in secondary structure by
\begin{itemize}
\item $f(\lambda)=-0.5$ if $\ell$ is a hairpin loop,
\item $f(\lambda)=1$ if $\ell$ is an interior loop,
\item $f(\lambda)=-5$ if $\ell$ is a multi-loop,
\end{itemize}
where $\lambda$ is a loop. The weight for a secondary structure $\delta$
accordingly is given by
\begin{equation}\label{E:loop}
f(\delta)=\sum_{\lambda \in \delta} f(\lambda).
\end{equation}

\begin{figure}[ht]
\begin{center}
\includegraphics[width=0.7\columnwidth]{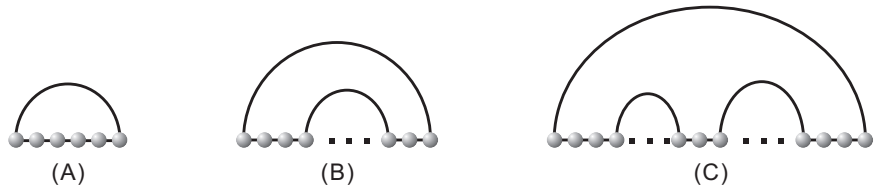}
\end{center}
\caption{\small Diagram representation of loop types: (A) hairpin loop,
(B) interior loop, (C) multi-loop. }
\label{F:Loops}
\end{figure}

Let ${\bf F}^*_0(z)$ and ${\bf F}_0(z)$ be the GFs obtained by setting
$t=e$ and $v=6/16$ in ${\bf F}^*(z,t,v)$ and ${\bf F}(z,t,v)$, where
$e$ is the Euler number. This means we find a suitable parameterization
which brings us back to a simple univariate GF.

%%%
%%%%%%%%%%%%%%%%%%%%%%%%%%%%%%%%%%%%%%%%%%%%%%%%%%%%%%%%%%%%%%%%%%%%%%%%%%%%%%%%
%%%
\begin{lemma}\label{L:eumel3}
The weight function of RNA secondary structures, ${\bf F}^*_0(z)$,
satisfies
\begin{equation}\label{E:sec_rec_loop}
{\bf F}^*_0(z)=\frac{6}{16} e^{0.5} z^2\frac{z}{1-z}+\frac{6}{16}e^{1} z^2 \left(\frac{1}{1-z}\right)^2
{\bf F}^*_0(z)+\frac{6}{16}e^{-5} z^2 \frac{\left({\bf F}_0^*(z)\frac{1}{1-z}\right)^2}{1-{\bf F}_0^*(z)\frac{1}{1-z}}\frac{1}{1-z}.
\end{equation}
and ${\bf F}^*(z)$ is uniquely determined by the above equation. Furthermore
\begin{equation}
{\bf F}_0(z)=\frac{1}{1-z}\frac{1}{1-{\bf F}_0^*(z)\frac{1}{1-z}}.
\end{equation}
\end{lemma}
%%%
%%%%%%%%%%%%%%%%%%%%%%%%%%%%%%%%%%%%%%%%%%%%%%%%%%%%%%%%%%%%%%%%%%%%%%%%%%%%%%%%
%%%

\begin{proof}
We first consider the GF ${\bf F}^*_0(z)$ whose coefficient of $z^n$
denotes the total weight of irreducible secondary structures over $n$
vertexes, where $(1,n)$ is an arc.
Thus it gives a term $6/16 z^2$. Isolated vertex lead to the term
$$
z^p \sum_{i=0}^{\infty} z^i=z^p\frac{1}{1-z},
$$
where $p$ denotes the minimum number of isolated vertexes to be inserted.
Depending on the types of loops formed by $(i,n)$, we have
\begin{itemize}
\item hairpin loops: $\frac{z}{1-z}$,
\item interior loops: ${\bf F}^*_0(z) \left(\frac{1}{1-z}\right)^2$,
\item multi-loops: there are at least two irreducible substructures,
      as well as isolated vertices, thus
$$
\frac{1}{1-z}\sum_{i=2}^{\infty}\left({\bf F}^*_0(z)\frac{1}{1-z}\right)^i=
\frac{\left({\bf F}_0^*(z)\frac{1}{1-z}\right)^2}{1-{\bf F}_0^*(z)\frac{1}{1-z}}\frac{1}{1-z}.
$$
\end{itemize}
We compute
$$
{\bf F}^*_0(z)=\frac{6}{16} \left(e^{0.5} z^2\frac{z}{1-z}+e^{1} z^2
\left(\frac{1}{1-z}\right)^2
{\bf F}^*_0(z)+e^{-5} z^2 \frac{\left({\bf F}_0^*(z)\frac{1}{1-z}\right)^2}
{1-{\bf F}_0^*(z)\frac{1}{1-z}}\frac{1}{1-z}\right),
$$
which establishes the recursion. The uniqueness of the solution as a
power series follows from the fact that each coefficient can evidently
be recursively computed.

An arbitrary secondary structure can be considered as a sequence of
irreducible substructure with certain intervals of isolated vertexes. Thus
$$
{\bf F}_0(z)=\frac{1}{1-z}\sum_{i=0}^{\infty}\frac{1}{1-z}{\bf F}^*_0(z)=
\frac{1}{1-z}\frac{1}{1-{\bf F}_0^*(z)\frac{1}{1-z}}.
$$
\end{proof}

%%%
%%%%%%%%%%%%%%%%%%%%%%%%%%%%%%%%%%%%%%%%%%%%%%%%%%%%%%%%%%%%%%%%%%%%%%%%%%%%%%%%
%%%
\begin{lemma}\label{L:eumel33}
${\bf F}^*_0(z)$ and ${\bf F}_0(z)$ have the same singular expansion.
\begin{equation}
{\bf f}^*_0(n) \sim \alpha n^{-\frac{3}{2}}\gamma^n, \quad \text{and} \quad
{\bf f}_0(n) \sim \beta n^{-\frac{3}{2}}\gamma^n,
\end{equation}
where $\alpha \approx 0.24$ and $\beta \approx 2.88$ are constants and $\gamma \approx 2.1673$
\end{lemma}
%%%
%%%%%%%%%%%%%%%%%%%%%%%%%%%%%%%%%%%%%%%%%%%%%%%%%%%%%%%%%%%%%%%%%%%%%%%%%%%%%%%%
%%%

\begin{proof}
Solving eq.~\ref{E:sec_rec_loop} we obtain a unique solution for
${\bf F}^*_0(z)$ whose coefficient are all positive. Observing the
dominant singularity of ${\bf F}^*_0(z)$ it is $\rho\approx 0.4614$.
${\bf F}_0(z)$ is a function of ${\bf F}^*_0(z)$ and we examine
the real root of minimal modulus of $1-{\bf F}^*_0(z)\frac{1}{1-z}=0$
is bigger than $\rho$. Then by the supercritical paradigm
\cite{Flajolet:07a} applying,
${\bf F}_0(z)$ and ${\bf F}^*_0(z)$ have identical exponential growth rates.
Furthermore, ${\bf F}^*_0(z)$ and ${\bf F}_0(z)$ have the same
sub-exponential factor $n^{-\frac{3}{2}}$, hence the lemma.
\end{proof}

%%%
%%%%%%%%%%%%%%%%%%%%%%%%%%%%%%%%%%%%%%%%%%%%%%%%%%%%%%%%%%%%%%%%%%%%%%%%%%%%%%
%%%
\begin{theorem}\label{T:eumel22}
Suppose an mfe secondary structure over an interval of length $m$ is
irreducible with probability $\mathbb{P}_0(m)=\frac{{\bf f}^*_0(m)}
{{\bf f}_0(m)}$, then the expected number of candidates for
sequences of lengths $n$ is
$$
\mathbb{E}_0(n)= \Theta(n^2)
$$
and furthermore, setting $\overline{\mathbb{E}}_g(n)=
\mathbb{E}_g(n)/\Omega(n)$, we have
$$
\overline{\mathbb{E}}_0(n) \sim  {\bf f}^*_0(n)/{\bf f}_0(n) \sim b,
$$
where $b=\alpha/\beta\approx 0.08$.
\end{theorem}

\begin{proof}
By Lemma \ref{L:eumel33} we have ${\bf f}^*_0(m)/{\bf f}_0(m)\sim b$ where
$b$ is a constant. The proof is completely analogous to that of Theorem
\ref{T:eumel1}.
\end{proof}

We show the distribution of $\mathbb{P}_{0}(m)$ and $\overline{\mathbb{E}}_{0}(n)$ in 
Figure~\ref{F:Prob1l}.
%%%
%%%%%%%%%%%%%%%%%%%%%%%%%%%%%%%%%%%%%%%%%%%%%%%%%%%%%%%%%%%%%%%%%%%%%%%%%%%%%%%%%%
%%%

\begin{figure}[ht]
\begin{center}
\includegraphics[width=0.7\columnwidth]{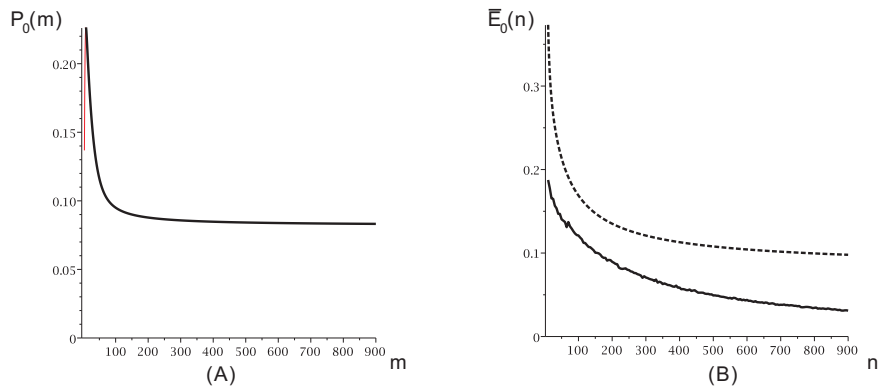}
\end{center}
\caption{\small The distribution of $\mathbb{P}_{0}(m)$ (A) and
$\overline{\mathbb{E}}_{0}(n)$ obtained by folding $100$
random sequences on the loop-based model (B)(solid), as well as
the theoretical expectation implied by Theorem~\ref{T:eumel22} (B)(dashed).
}
\label{F:Prob1l}
\end{figure}

\section*{Results and Discussion}

In this paper we quantify the effect of sparsification of the particular
decomposition rule $\Lambda^*$. This rule splits and interval and thereby
separates concatenated substructures.
The sparsification of $\Lambda^*$ alone is claimed to provide a speed up
of up to a linear factor of the DP-folding of RNA secondary structures
\cite{spar:07}. A similar conclusion is drawn in \cite{Backofen:08} where the
sparsification of RNA-RNA interaction structures is shown to experience
also a linear reduction in time complexity.
Both papers \cite{spar:07,Backofen:08} base their conclusion on the validity
of the polymer-zeta property discussed in Section~{\bf Sparsification}.

For the folding of pseudoknot structures there may however exist
non-sparsifiable rules in which case the overall time complexity is not
reduced.
The key object here is the {\it set of candidates} and we provide an
analysis of $\Lambda^*$-candidates by combinatorial means. In general,
the connection between candidates, i.e.~unions of disjoint intervals and
the combinatorics of structures is actually established by the algorithm
itself via backtracking: at the end of the DP-algorithm a structure is
being generated that realizes the previously computed energy as
mfe-structure. This connects intervals and sub-structures.

So, does polymer-zeta apply in the context of RNA structures?
In fact polymer-zeta would follow {\it if} the intervals in question
are distributed as in uniformly sampled structures. This however, is
far from reasonable, due to the fact that the mfe-algorithm
deliberately designs some mfe structure over the given interval.
What the algorithm produces is in fact antagonistic to uniform sampling.
We here wish to acknowledge the help of one anonymous referee in
clarifying this point.

Our results clearly show that the polymer-zeta property, i.e.~the
probability of an irreducible structure over an interval of
length $m$ satisfies a formula of the form
\begin{equation}\label{E:mmn}
\mathbb{P}(\text{there exists an irreducible structure over $[1,m]$} ) =
b \, m^{1+c}, \quad \text{\rm where }\ b,c>0.
\end{equation}
does not apply for RNA structures. The theoretical findings from
self-avoiding walks \cite{Kabakcioglu:05} unfortunately do not allow
to quantify the expected number of candidates of the $\Lambda^*$-rule
in RNA folding.

That the polymer-zeta property does not hold for RNA has also
been observed in the context of the limit distribution of the
5'-3' distances of RNA secondary structures \cite{Hillary:11}.
Here it is observed that long arcs, to be precise arcs of
lengths $O(n)$ {\it always} exist.
This is of course a contradiction to eq.~(\ref{E:mmn}).

The key to quantification of the expected number of candidates is the
singularity analysis of a pair of energy-filtered GF, namely that of
a class of structures and that of the subclass of all such
structures that are irreducible.
We show that for various energy models the singular
expansions of both these functions are essentially {\it equal}--modulo
some constant. This implies that the expected number of candidates
is $\Theta(n^2)$ and all constants can explicitly be computed from
a detailed singularity analysis. The good news is that depending on
the energy model, a significant constant reduction, around $95\%$ can be
obtained. This is in accordance with data produced in \cite{Backofen:11}
for the mfe-folding of random sequences. There a reduction by $98\%$ is
reported for sequences of length $\ge 500$.

Our findings are of relevance for numerous results, that are formulated
in terms of sizes of candidate sets \cite{Backofen:10}. These can now
be quantified. It is certainly of interest to devise a full fledged
analysis of the loop-based energy model. While these computations are
far from easy our framework shows how to perform such an analysis.

Using the paradigm of gap-matrices Backofen has shown \cite{Backofen:10}
that the sparsification of the DP-folding of RNA pseudoknot structures
exhibits additional instances, where sparsification can be applied, see
Fig.~\ref{F:spar_g} (B).
Our results show that the expected number of candidates is $\Theta(n^2)$,
where the constant reduction is around $90\%$. This is in fact very good
new since the sequence length in the context of RNA pseudoknot structure
folding is in the order of hundreds of nucleotides. So sparsification of
further instances does have an significant impact on the time complexity
of the folding.

%%%
%%%%%%%%%%%%%%%%%%%%%%%%%%%%%%%%%%%%%%%%%%%%%%%%%%%%%%%%%%%%%%%%%%%%%%%%%%%
%%%
\section*{Proofs} \label{S:proof}
%%%
%%%%%%%%%%%%%%%%%%%%%%%%%%%%%%%%%%%%%%%%%%%%%%%%%%%%%%%%%%%%%%%%%%%%%%%%%%%%
%%%

In this section, we prove Lemma~\ref{L:irr_rec} and Theorem~\ref{T:irr}.

{\bf Proof for Lemma~\ref{L:irr_rec}:}
let ${\bf D}(z,u)$ and ${\bf D}^*(z,u)$ be the bivariate GF
${\bf D}(z,u)=\sum_{n\geq 0}\sum_{g=0}^{\lfloor\frac{n}{2}\rfloor}{\bf d}_g(n)z^nu^g$,
and ${\bf D}^*(z,u)=\sum_{n\geq 1}\sum_{g=0}^{\lfloor\frac{n}{2}\rfloor}
{\bf d}^*_g(n)z^nu^g$.
Suppose a structure contains exactly $j$ irreducible structures, then
\begin{equation}
{\bf{D}}(z,u)=\sum_{j\geq 0}{\bf{R}}(z,u)^j=\frac{1}{1-{\bf{R}}(z,u)}
\end{equation}
and
\begin{equation}
{\bf D}^*_g(z)=[u^g]{\bf D}^*(z,u)=-[u^g]\frac{1}{{\bf D}(z,u)},\quad g\geq 1,
\end{equation}
as well as ${\bf D}^*_0(z)=1-[u^0]\frac{1}{{\bf D}(z,u)}$.
Let ${\bf F}(z,u)=\sum_{n\geq 0}\sum_{g\geq 0}{\bf f}_g(n)z^nu^g=
\frac{1}{{\bf D}(z,u)}$. Then ${\bf{F}}(z,u){\bf{D}}(z,u)=1$, whence
for $g\geq 1$,
\begin{equation}
\sum_{g_1=0}^g{\bf{F}}_{g_1}(z){\bf{D}}_{g-g_1}(z)
=[u^g]{\bf{F}}(z,u){\bf{D}}(z,u)=0,
\end{equation}
and ${\bf{F}}_{0}(z){\bf{D}}_{0}(z)=1$, where
${\bf{F}}_g(z)=\sum_{n\geq 0}{\bf f}_g(n)z^n=
[u^g]{\bf{F}}(z,u)=[u^g]\frac{1}{{\bf{D}}(z,u)}$.
Furthermore, we have ${\bf{F}}_0(z)=\frac{1}{{\bf{D}}_{0}(z)}$ and
\begin{equation}
{\bf {F}}_g(z)=-\frac{\sum_{g_1=0}^{g-1}{\bf{F}}_{g_1}(z)
{\bf{D}}_{g-g_1}(z)}
{{\bf{D}}_{0}(z)},\quad g\geq 1,
\end{equation}
which implies ${\bf D}^*_0(z)=1-{\bf{F}}_0(z)=1-\frac{1}{{\bf{D}}_{0}(z)}$ and
\begin{equation}
{\bf D}^*_g(z)=-{\bf{F}}_g(z)=-\frac{({\bf D}^*_0(z)-1){\bf{D}}_{g}(z)
+\sum_{g_1=1}^{g-1}
{\bf D}^*_{g_1}(z){\bf{D}}_{g-g_1}(z)}
{{\bf{D}}_{0}(z)}.
\end{equation}

{\bf Proof for Theorem~\ref{T:irr}}
Let $[n]_k$ denotes the set of compositions of $n$ having $k$ parts,
i.e.~for$\sigma\in [n]_k$ we have $\sigma=(\sigma_1,\ldots,\sigma_k)$
and $\sum_{i=1}^k\sigma_i=n$.\\
{\it Claim.}
\begin{equation} \label{E:irr_exp}
{\bf D}^*_{g+1}(z)=\frac{{\bf{D}}_{g+1}(z)}{{\bf{D}}_{0}(z)^2}
+\sum_{j=0}^{g-1}\frac{(-1)^{g+2-j}}{{\bf{D}}_{0}(z)^{g+2-j}}
\left(\sum_{\sigma\in[g+1]_{g+1-j}}\prod_{i=1}^{g+1-j}
{\bf{D}}_{\sigma_i}(z)\right).
\end{equation}
We shall prove the claim by induction on $g$. For $g=1$ we have
\begin{equation}
{\bf D}^*_1(x)=\frac{{\bf{D}}_{1}(z)}
{\left({\bf{D}}_{0}(z)\right)^2},
\end{equation}
whence eq.~(\ref{E:irr_exp}) holds for $g=1$.
By induction hypothesis, we may now assume that for $j\leq g$,
eq.~(\ref{E:irr_exp}) holds. According to Lemma~\ref{L:irr_rec},
we have
\begin{eqnarray*}
{\bf D}^*_{g+1}(z) &=&-\frac{({\bf D}^*_0(z)-1){\bf{D}}_{g+1}(z)
+\sum_{g_1=1}^{g}
{\bf D}^*_{g_1}(z){\bf{D}}_{g+1-g_1}(z)}
{{\bf{D}}_{0}(z)}\\
&=&\frac{{\bf{D}}_{g+1}(z)}{{\bf{D}}_{0}(z)^2}-
\sum_{g_1=1}^{g}\left(\frac{{\bf{D}}_{g_1}(z)}{{\bf{D}}_{0}(z)^{3}}
+\sum_{j=0}^{g_1-2}\frac{(-1)^{g_1+1-j}}{{\bf{D}}_{0}(z)^{g_1+2-j}}
\left(\sum_{\sigma\in[g_1]_{g_1-j}}\prod_{i=1}^{g_1-j}
{\bf{D}}_{\sigma_i}(z)\right)\right){\bf{D}}_{g+1-g_1}(z).
\end{eqnarray*}
We next observe
\begin{equation}
-\sum_{g_1=1}^{g}\frac{{\bf{D}}_{g_1}(z)}
{{\bf{D}}_{0}(z)^{3}}{\bf{D}}_{g+1-g_1}(z)=\frac{(-1)^{g+2-(g-1)}}
{{\bf{D}}_{0}(z)^{g+2-(g-1)}}
\left(\sum_{\sigma'\in[g+1]_{g+1-(g-1)}}\prod_{i=1}^{g+1-(g-1)}
{\bf{D}}_{\sigma_i'}(z)\right),
\end{equation}
and setting $h=g_1-j$ we obtain,
\begin{eqnarray*}
&&-\sum_{g_1=1}^{g}\sum_{j=0}^{g_1-2}\frac{(-1)^{g_1+1-j}}{{\bf{D}}_{0}(z)^{g_1+2-j}}
\left(\sum_{\sigma\in[g_1]_{g_1-j}}\prod_{i=1}^{g_1-j}
{\bf{D}}_{\sigma_i}(z)\right){\bf{D}}_{g+1-g_1}(z)\\
&=&\sum_{g_1=1}^{g}\sum_{h=2}^{g_1}\frac{(-1)^{h+2}}
{{\bf{D}}_{0}(z)^{h+2}}
\left(\sum_{\sigma\in[g_1]_{h}}\prod_{i=1}^{h}
{\bf{D}}_{\sigma_i}(z)\right){\bf{D}}_{g+1-g_1}(z)\\
&=&\sum_{h=2}^{g}\frac{(-1)^{h+2}}
{{\bf{D}}_{0}(z)^{h+2}}\left(\sum_{g_1=h}^{g}
\left(\sum_{\sigma\in[g_1]_{h}}\prod_{i=1}^{h}
{\bf{D}}_{\sigma_i}(z)\right){\bf{D}}_{g+1-g_1}(z)\right)\\
&=&\sum_{h=2}^{g}\frac{(-1)^{h+2}}
{{\bf{D}}_{0}(z)^{h+2}}
\left(\sum_{\sigma'\in[g+1]_{h+1}}\prod_{i=1}^{h+1}
{\bf{D}}_{\sigma_i'}(z)\right)
\end{eqnarray*}
and setting $j=g-h$
\begin{equation*}
=\sum_{j=0}^{g-2}\frac{(-1)^{g+2-j}}
{{\bf{D}}_{0}(z)^{g+2-j}}
\left(\sum_{\sigma'\in[g+1]_{g+1-j}}\prod_{i=1}^{g+1-j}
{\bf{D}}_{\sigma_i'}(z)\right).
\end{equation*}
Consequently, the Claim holds for any $g\ge 1$.

For any $g\geq 1$, we have \cite{rnag3}
$$
\mathbf{D}_g(z) = \frac{1}{z^2-z+1}
                  \frac{{\bf P}_g(u)}{{(1-4u)^{3g-1/2}}}, \qquad
\mathbf{D}_0(z) = \frac{1}{z^2-z+1} \frac{2}{(1+\sqrt{1-4u})},
$$
where ${\bf P}_g(u)$ is a polynomial with integral coefficients of
degree at most $(3g-1)$, ${\bf P}_g(1/4) \neq 0$, $[u^{2g}]{\bf P}_g(u)
\neq 0$ and $[u^h]{\bf P}_g(u) = 0$ for $0\le h\le 2g-1$.
Let $u=\frac{z^2}{(z^2-z+1)^2}$, the Claim provides in this context
the following interpretation of
${\bf D}^*_{g}(z)$
\begin{equation}
\frac{1}{z^2-z+1}{\bf D}^*_{g}(z)=\frac{{\bf P}_g(u)}{(1-4u)^{3g-1/2}}
\left(\frac{1+\sqrt{1-4u}}{2}\right)^2+\sum_{j=0}^{g-2}
\left(-\frac{1+\sqrt{1-4u}}{2}\right)^{g+1-j}
\frac{\sum_{\sigma\in[g]_{g-j}}\prod_{i=1}^{g-j}
{\bf P}_{\sigma_i}(u)}{(1-4u)^{3g-\frac{g-j}{2}}},
\end{equation}
and
\begin{eqnarray*}
&&\sum_{j=0}^{g-2}
\left(-\frac{1+\sqrt{1-4u}}{2}\right)^{g+1-j}
\frac{\sum_{\sigma\in[g]_{g-j}}\prod_{i=1}^{g-j}
{\bf P}_{\sigma_i}(u)}{(1-4u)^{3g-\frac{g-j}{2}}}\\
&=&\sum_{j=0}^{g-2}\sum_{k=0}^{g+1-j}
\left(-\frac{1}{2}\right)^{g+1-j}\binom{g+1-j}{k}
\frac{\sum_{\sigma\in[g]_{g-j}}\prod_{i=1}^{g-j}
{\bf P}_{\sigma_i}(u)}{(1-4u)^{3g-\frac{g-j+k}{2}}}\\
&=&\sum_{j=0}^{g-2}\sum_{s=g-j}^{2g+1-2j}
\left(-\frac{1}{2}\right)^{g+1-j}\binom{g+1-j}{s-g+j}
\frac{\sum_{\sigma\in[g]_{g-j}}\prod_{i=1}^{g-j}
{\bf P}_{\sigma_i}(u)}{(1-4u)^{3g-\frac{s}{2}}}.
\end{eqnarray*}
As $0\leq j\leq g-2$ and $g-j\leq s\leq 2g+1-2j$, we have $s\geq 2$.
Consequently we arrive at
\begin{equation}
\frac{1}{z^2-z+1}{\bf D}^*_g(z)=\frac{{\bf U}_g(u)}{(1-4u)^{3g-1/2}}+
\frac{{\bf V}_g(u)}{(1-4u)^{3g-1}},
\end{equation}
where
\begin{eqnarray*}
{\bf U}_g(u)&=&\frac{{\bf P}_g(u)}{4}+\frac{{\bf P}_g(u)(1-4u)}{4}\\
&+&
\sum_{j=0}^{g-2}\sum_{g-j\leq s\leq 2g+1-2j\atop{s\text{ is odd}}}
\left(-\frac{1}{2}\right)^{g+1-j}\binom{g+1-j}{s-g+j}
\left(\sum_{\sigma\in[g]_{g-j}}\prod_{i=1}^{g-j}
{\bf P}_{\sigma_i}(u)\right)(1-4u)^{\frac{s-1}{2}},
\end{eqnarray*}
and
\begin{eqnarray*}
{\bf V}_g(u)&=&\frac{{\bf P}_g(u)}{2}+\left(-\frac{1}{2}\right)^{3}
\left(\sum_{\sigma\in[g]_{2}}\prod_{i=1}^{2}
{\bf P}_{\sigma_i}(u)\right)+3\left(-\frac{1}{2}\right)^{3}
\left(\sum_{\sigma\in[g]_{2}}\prod_{i=1}^{2}
{\bf P}_{\sigma_i}(u)\right)(1-4u)\\
&+&
\sum_{j=0}^{g-3}\sum_{g-j\leq s\leq 2g+1-2j\atop{s\text{ is even}}}
\left(-\frac{1}{2}\right)^{g+1-j}\binom{g+1-j}{s-g+j}
\left(\sum_{\sigma\in[g]_{g-j}}\prod_{i=1}^{g-j}
{\bf P}_{\sigma_i}(u)\right)(1-4u)^{\frac{s-2}{2}}.
\end{eqnarray*}
We have for $\sigma\in[g]_{k}$, $k\geq1$
\begin{equation*}
[u^h]\left(\sum_{\sigma\in[g]_{k}}\prod_{i=1}^{k}
{\bf P}_{\sigma_i}(u)\right)=\sum_{\sigma\in[g]_{k}}
\prod_{i=1}^{k}[u^{h_i}]{\bf P}_{\sigma_i}(u),
\end{equation*}
where $\sum_{i=1}^k h_i=h$, $h_i\geq 0$.
Then we obtain that
\begin{equation}
[u^h]\left(\sum_{\sigma\in[g]_{k}}\prod_{i=1}^{k}
{\bf P}_{\sigma_i}(u)\right)=0, \quad 0\leq h\leq 2g-1.
\end{equation}
Since $[u^{h_i}]{\bf P}_{\sigma_i}(u)=0$, $h_i\leq 2\sigma_i-1$,
$[u^{2\sigma_i}]{\bf P}_{\sigma_i}(u)\neq 0$ and $\sum_{i=1}^k\sigma_i=g$.
Thus for $0\leq h\leq 2g-1$,
\begin{equation}
[u^h]{\bf U}_g(u)=0\quad\text{and}\quad [u^h]{\bf V}_g(u)=0.
\end{equation}
As shown in \cite{rnag3} we have
\begin{equation}\label{E:q_g41}
{\bf P}_g(1/4) = {\frac {\Gamma  \left( g-1/6 \right)
           \Gamma  \left( g+1/2\right) \Gamma
           \left( g+1/6 \right) {9}^{g}{4}^{-g}}{6{\pi }^{3/2}
           \Gamma \left( g+1 \right)}}
\end{equation}
and we obtain ${\bf U}_g(1/4)={\bf P}_g(1/4)/4$. Furthermore,
$$
{\bf V}_g(1/4)=\frac{{\bf P}_g(1/4)}{2}+\left(-\frac{1}{2}\right)^{3}
\left(\sum_{\sigma\in[g]_{2}}\prod_{i=1}^{2}
{\bf P}_{\sigma_i}(1/4)\right)
=\frac{1}{8}\left(4{\bf P}_g(1/4)-\sum_{j=1}^{g-1}{\bf P}_j(1/4)
{\bf P}_{g-j}(1/4)\right)\neq 0.
$$
We can recruit the computation of \cite{rnag3} in order to observe
$4{\bf P}_g(1/4)-\sum_{j=1}^{g-1}{\bf P}_j(1/4)
{\bf P}_{g-j}(1/4)\neq 0$.
In order to compute the bivariate GF, ${\bf E}^*_g(z,t)$,
we only need to replace in eq.~(\ref{E:irr_exp})
${\bf D}_g(z)$ by ${\bf E}_g(z,t)$ and the proof is completely
analogous.
\bigskip

%%%%%%%%%%%%%%%%%%%%%%%%%%%%%%%%
\section*{Author's contributions}
Fenix W.D. Huang and Christian M. Reidys contributed equally to
research and manuscript.

%%%%%%%%%%%%%%%%%%%%%%%%%%%
\section*{Acknowledgements}
  \ifthenelse{\boolean{publ}}{\small}{}
  We want to thank Thomas J.X.~Li for discussions and comments.
  We want to thank an anonymous referee for pointing out an incorrect
assumption of first version of this paper. His comments have led to
a much improved version of the paper.

%%  Bmc_article.bst  will be used to                       %%
%%  create a .BBL file for submission, which includes      %%
%%  XML structured for BMC.                                %%
%%  After submission of the .TEX file,                     %%
%%  you will be prompted to submit your .BBL file.         %%

\newpage
{\ifthenelse{\boolean{publ}}{\footnotesize}{\small}
 \bibliographystyle{bmc_article}  % Style BST file
  \bibliography{spar_article} }     % Bibliography file (usually '*.bib' )

%%%%%%%%%%%

\ifthenelse{\boolean{publ}}{\end{multicols}}{}

%%%%%%%%%%%%%%%%%%%%%%%%%%%%%%%%%%%
%%                               %%
%% Figures                       %%
%%                               %%
%% NB: this is for captions and  %%
%% Titles. All graphics must be  %%
%% submitted separately and NOT  %%
%% included in the Tex document  %%
%%                               %%
%%%%%%%%%%%%%%%%%%%%%%%%%%%%%%%%%%%

%%
%% Do not use \listoffigures as most will included as separate files

\section*{Figures}
  \subsection*{Figure 1 - RNA structures as planar graphs and diagrams}
  (A) an RNA secondary structure and (B) an RNA pseudoknot structure.

  \subsection*{Figure 2 - Sparsification of secondary structure folding}
Suppose the optimal solution $L_{i,j}$ is obtained from the optimal solutions
$L_{i,k}$, $L_{k+1,q}$ and $L_{q+1,j}$. Based on the recursions of the secondary
structures, $L_{i,k}$ and $L_{k+1,q}$ produce an optimal solution of
$L_{i,q}$. Similarly, $L_{k+1,q}$ and $L_{q+1,j}$ produce an optimal solution
of $L_{k+1,j}$.
Now, in order to obtain an optimal solution of $L_{i,j}$ it is sufficient to
consider either the grouping $L_{i,q}$ and $L_{q+1,j}$ or $L_{i,k}$ and
$L_{k+1,j}$.

  \subsection*{Figure 3 - What sparsification can and cannot prune}
What sparsification can and cannot prune: (A) and (B) are two computation
paths yielding the same optimal solution. Sparsification reduces the
computation to path (A) where $S_{i,k_1}$ is irreducible.
(C) is another computation path with distinct leftmost irreducible over
a different interval, hence representing a new candidate that cannot be
reduced to (A) by the sparsification.

  \subsection*{Figure 4 - Sparsification}
Sparsification: $L_v$ is alternatively realized via
$L_{v_1}$ and $L_{v_2'}$, or $L_{v_1'}$ and $L_{v_3}$. Thus it is sufficient
to only consider one of the computation paths.

  \subsection*{Figure 5 -  The recursion solving the optimal solution for
secondary structures}
The recursion solving the optimal solution for
secondary structures.

  \subsection*{Figure 6 - Decomposition rules for pseudoknot structures of fixed genus}
 (A) three decompositions via the rule $\Lambda^*$, which is $s$-compatible to
    itself.
    We show that for $\Lambda^*$ we obtain a linear reduction in time
    complexity.
(B) three decomposition rules $\Lambda_1,\Lambda_2,\Lambda_3$ where
    $\Lambda_2,\Lambda_3$ are $s$-compatible to $\Lambda_1$.
    A quantification of the candidate set is not implied by the
    polymer-zeta property.
(C) three decomposition rules $\Lambda_1,\Lambda_2,\Lambda_3$ where
    $\Lambda_2,\Lambda_3$ are not $s$-compatible to $\Lambda_1$.

  \subsection*{Figure 7 - RNA structures and diagram representation  }
   A diagram over $\{1,\ldots,40\}$. The arcs $(1,21)$ and
$(11,33)$ are crossing and the dashed arc $(9,10)$ is a $1$-arc which is not allowed.
This structure contains $3$ stacks with length $7$, $4$ and $6$, from left to right
respectively.

  \subsection*{Figure 8 - Irreducibility relative to a decomposition rule }
the rule $\Lambda^*$ splitting $S_{i,j}$ to $S_{i,k}$ and
$S_{k+1,j}$, $S_{1,40}$ is not $\Lambda^*$-irreducible, while $S_{1,25}$ and
$S_{28,40}$ are.
However, for the decomposition rule $\Lambda_2$, which removes the
outmost arc, $S_{28,40}$ is not $\Lambda_2$-irreducible while $S_{1,25}$ is.

  \subsection*{Figure9 -The expected number of candidates for secondary and $1$-structures
$\overline{\mathbb{E}}_{0}(n)$ and $\overline{\mathbb{E}}_{1}(n)$}
we compute the expected number of candidates obtained by folding
$100$ random sequences for secondary structures (A)(solid)
and $1$-structures (B)(solid). We also display the theoretical expectations
implied by Theorem~\ref{T:eumel1} (A)(dashed) and (B)(dashed).

  \subsection*{Figure10- The probability distribution of $\mathbb{P}_{0}(m)$ and
$\mathbb{P}_{1}(m)$}
The probability distribution of $\mathbb{P}_{0}(m)$ (A) and
$\mathbb{P}_{1}(m)$ (B)

  \subsection*{Figure11 -Diagram representation of loop types}
(A) hairpin loop, (B) interior loop, (C) multi-loop.

  \subsection*{Figure12 -The distribution of $\mathbb{P}_{0}(m)$ (A) and
$\overline{\mathbb{E}}_{0}(n)$}
The distribution of $\mathbb{P}_{0}(m)$ (A) and
$\overline{\mathbb{E}}_{0}(n)$ obtained by folding $100$
random sequences on the loop-based model (B)(solid), as well as
the theoretical expectation implied by Theorem~\ref{T:eumel22} (B)(dashed).

%%%%%%%%%%%%%%%%%%%%%%%%%%%%%%%%%%%
%%                               %%
%% Tables                        %%
%%                               %%
%%%%%%%%%%%%%%%%%%%%%%%%%%%%%%%%%%%

%% Use of \listoftables is discouraged.
%%

%%%%%%%%%%%%%%%%%%%%%%%%%%%%%%%%%%%
%%                               %%
%% Additional Files              %%
%%                               %%
%%%%%%%%%%%%%%%%%%%%%%%%%%%%%%%%%%%

\end{bmcformat}
\end{document}